\documentclass[12pt]{amsart}
\usepackage[left=1in,right=1in,top=1.1in,bottom=1.1in]{geometry}
 
\usepackage{amssymb,amsmath,amsthm,verbatim,enumerate,commath}
\usepackage{xcolor}
\usepackage{url}
\usepackage{mathrsfs,mathtools,dsfont,bbm}
\usepackage{tikz-cd}
\usepackage[all]{xy}
  \SelectTips{cm}{10}
  \everyxy={<2.5em,0em>:}
\usepackage{tikz}
\usepackage{picinpar} 

\usepackage{microtype}
\usepackage[normalem]{ulem}

\usepackage{ifpdf}
  \ifpdf
    \usepackage[pagebackref,linktocpage]{hyperref} 
    \hypersetup{
colorlinks=true,		
linkcolor=red,		
citecolor=blue,		
urlcolor=cyan		
}
  \else
    
    \newcommand{\href}[2]{#2}
  \fi

\newtheorem{theorem}{Theorem}[section]
\newtheorem{lemma}[theorem]{Lemma}
\newtheorem{proposition}[theorem]{Proposition}

\newtheorem{conjecture}[theorem]{Conjecture}

\theoremstyle{definition}
\newtheorem{definition}[theorem]{Definition}
\newtheorem{example}[theorem]{Example}

\theoremstyle{remark}
\newtheorem{remark}[theorem]{Remark}

\newtheorem{step}{Step}

\numberwithin{equation}{section}

 \DeclareFontFamily{U}{manual}{}
 \DeclareFontShape{U}{manual}{m}{n}{ <->  manfnt }{}
 \newcommand{\manfntsymbol}[1]{%
    {\fontencoding{U}\fontfamily{manual}\selectfont\symbol{#1}}}

\makeatletter
   \@addtoreset{section}{part}
   \@addtoreset{equation}{section}

    {\hspace*{\fill}$\lrcorner$\endgraf\endgroup\end{trivlist}}
\makeatother

  \DeclareFontFamily{OT1}{pzc}{}
  \DeclareFontShape{OT1}{pzc}{m}{it}{<-> s * [1.100] pzcmi7t}{}
  \DeclareMathAlphabet{\mathpzc}{OT1}{pzc}{m}{it}

\newif\ifhascomments \hascommentstrue
\ifhascomments
  \newcommand{\jason}[1]{{\color{red}[[\ensuremath{\bigstar\bigstar\bigstar} #1]]}}
  \newcommand{\matt}[1]{{\color{red}[[\ensuremath{\spadesuit\spadesuit\spadesuit} #1]]}}
  \newcommand{\fei}[1]{{\color{red}[[\ensuremath{\clubsuit\clubsuit\clubsuit} #1]]}}
\else
  \newcommand{\jason}[1]{}
  \newcommand{\matt}[1]{}
  \newcommand{\fei}[1]{}
\fi

\newcommand{\Qbar}{\overline{\mathbb{Q}}}
\renewcommand{\AA}{\mathbb{A}}

\DeclareMathOperator{\Aut}{\ensuremath{\mathcal{A}\kern-.125em\mathpzc{ut}}}

\newcommand{\CC}{\mathbb C}

\newcommand{\dra}{\dashrightarrow}
\renewcommand{\emptyset}{\varnothing}

\DeclareMathOperator{\Endo}{\ensuremath{\mathcal{E}\kern-.125em\mathpzc{nd}}}

\newcommand{\OO}{\mathcal{O}}

\newcommand{\bP}{\mathbb{P}}

\DeclareMathOperator{\Hom}{\ensuremath{\mathcal{H}\kern-.125em\mathpzc{om}}}

\newcommand{\NN}{\mathbb N}

\newcommand{\PP}{\mathbb{P}}

\newcommand{\QQ}{\mathbb Q}

\renewcommand{\setminus}{\smallsetminus}

\newcommand{\ZZ}{\mathbb{Z}}

 \def\are[#1]{\ar[#1]^{\txt{\'et}}}
 \def\areh[#1]{\ar[#1]|{\txt{$H$-eq}}^{\txt{\'et}}}
 \def\ars[#1]{\ar@{->>}[#1]}
 \newcommand{\dplus}{\ar@{}[d]|{\mbox{$\oplus$}}}
 \newcommand{\dtimes}{\ar@{}[d]|{\mbox{$\times$}}}

\newcommand{\KK}{{\mathbb K}}

\newcommand{\arxiv}[1]{\href{https://arxiv.org/abs/#1}{{\tt arXiv:#1}}}
\newcommand{\doi}[1]{\href{https://doi.org/#1}{{\tt doi:#1}}}
\newcommand{\ol}[1]{\overline{#1}}
\newcommand{\den}{\overline{d}}
\newcommand{\cO}{{\mathcal O}}
\newcommand{\bn}{\mathbf{n}}

\newcommand{\bz}{\mathbf{z}}
\newcommand{\bT}{\mathbf{T}}

\title[Height Gap Conjectures, $D$-Finiteness, and Weak DML]{Height Gap Conjectures, $D$-Finiteness, and Weak Dynamical Mordell--Lang}

\author{Jason P. Bell}
\address{University of Waterloo, Department of Pure Mathematics, Waterloo, Ontario, N2L 3G1, Canada}
\email{\href{mailto:jpbell@uwaterloo.ca}{jpbell@uwaterloo.ca}}

\author{Fei Hu}
\address{University of Waterloo, Department of Pure Mathematics, Waterloo, Ontario, N2L 3G1, Canada}
\email{\href{mailto:f8hu@uwaterloo.ca}{f8hu@uwaterloo.ca}}

\author{Matthew Satriano}
\address{University of Waterloo, Department of Pure Mathematics, Waterloo, Ontario, N2L 3G1, Canada}
\email{\href{mailto:msatrian@uwaterloo.ca}{msatrian@uwaterloo.ca}}

\thanks{The first and third authors were partially supported by Discovery Grants from the National Science and Engineering Research Council of Canada. The second author was partially supported by a postdoctoral fellowship at the University of Waterloo.}

\keywords{Weil height, Dynamical Mordell--Lang, rational maps, $D$-finite power series}

\subjclass[2010]{11G50, 37P55, 14E05}

\begin{document}

\begin{abstract}
In previous work, the first author, Ghioca, and the third author introduced a broad dynamical framework giving rise to many classical sequences from number theory and algebraic combinatorics. Specifically, these are sequences of the form $f(\Phi^n(x))$, where $\Phi\colon X\dasharrow X$ and $f\colon X\dasharrow\mathbb{P}^1$ are rational maps defined over $\overline{\mathbb{Q}}$ and $x\in X(\overline{\mathbb{Q}})$ is a point whose forward orbit avoids the indeterminacy loci of $\Phi$ and $f$. They conjectured that if the sequence is infinite, then $\limsup \frac{h(f(\Phi^n(x)))}{\log n} > 0$. They also made a corresponding conjecture for $\liminf$ and showed that it implies the Dynamical Mordell--Lang Conjecture. In this paper, we prove the $\limsup$ conjecture as well as the $\liminf$ conjecture away from a set of density $0$. As applications, we prove results concerning the growth rate of coefficients of $D$-finite power series as well as the Dynamical Mordell--Lang Conjecture up to a set of density $0$.
\end{abstract}

\maketitle
\tableofcontents


\section{Introduction}
\label{sec:intro}

In \cite{BGS}, the authors introduced a broad dynamical framework giving rise to many classical sequences from number theory and algebraic combinatorics. In particular, this construction yields all sequences whose generating functions are $D$-finite, i.e., those satisfying homogeneous linear differential equations with rational function coefficients.  This class, in turn, contains all hypergeometric series (see, e.g., \cite{WZ, G09}), all series related to integral factorial ratios \cite{Bober09, Sound}, generating functions for many classes of lattice walks \cite{DHRS18}, diagonals of rational functions \cite{Lipshitz88}, algebraic functions \cite{Lipshitz89}, generating series for the cogrowth of many finitely presented groups \cite{GP}, as well as generating functions of numerous classical combinatorial sequences (see Stanley \cite[Chapter 6]{Stan} and the examples therein).  In \cite{BGS}, they also stated the so-called $\limsup$ and $\liminf$ Height Gap Conjectures, which if true, would imply both the Dynamical Mordell--Lang Conjecture as well as results concerning the growth rate of coefficients of $D$-finite power series. The goal of this paper is to prove a uniform version of the $\limsup$ Height Gap Conjecture and to prove the $\liminf$ version away from a set of density zero. Consequently, we obtain applications to $D$-finite power series and a weak version of the Dynamical Mordell--Lang Conjecture.

To state our results, we fix the following notation. Throughout, we let $\NN$ (resp.~$\ZZ^+$) denote the set of all non-negative (resp.~positive) integers.
Let $h(\cdot)$ denote the absolute logarithmic Weil height function.
We refer the reader to \cite[Chapter~2]{BG06} and \cite[Chapter~3]{GTM241} for the main properties of height functions. 
Given an arbitrary rational map $g$, let $I_g$ denote its indeterminacy locus.
If $\Phi$ is a rational self-map of a quasi-projective variety $X$ defined over $\Qbar$, then we let $X_{\Phi}(\Qbar)$ denote the subset of points $x\in X(\Qbar)$ such that for all $n\in \NN$, the $n$-th iterate $\Phi^{n}(x)$ avoids $I_\Phi$; for such an $x\in X_\Phi(\Qbar)$, we let $\cO_\Phi(x)$ denote its forward orbit under $\Phi$. Lastly, if $f\colon X\dra \bP^1$ is a rational function, let $X_{\Phi,f}(\Qbar)\subseteq X_\Phi(\Qbar)$ be the subset of points $x$ with $I_f\cap\cO_\Phi(x)=\varnothing$.

The following conjecture was introduced in \cite{BGS}.

\begin{conjecture}[{$\limsup$ Height Gap Conjecture, cf.~\cite[Conjecture~1.4]{BGS}}]
\label{conj:limsup}
Let $X$ be a quasi-projective variety, let $\Phi\colon X\dra X$ be a rational self-map, and let $f\colon X\dra \bP^1$ be a non-constant rational function, all defined over $\Qbar$.
Then for any $x\in X_{\Phi,f}(\Qbar)$, either $f(\cO_\Phi(x))$ is finite, or
\[
\limsup_{n\to\infty} \frac{h(f(\Phi^n(x)))}{\log n}>0.
\]
\end{conjecture}

Our first main result is a simple proof of Conjecture~\ref{conj:limsup}. This generalizes \cite[Theorem~1.3]{BGS}, which handled the case where $\Phi$ and $f$ are morphisms.

\begin{theorem}[$\limsup$ Height Gaps]
\label{thm:limsup}
Conjecture~\ref{conj:limsup} is true.
\end{theorem}

In \cite{BGS}, the authors also introduced the following conjecture and showed that it implies Dynamical Mordell--Lang.

\begin{conjecture}[{$\liminf$ Height Gap Conjecture, cf.~\cite[Conjecture 1.6]{BGS}}]
\label{conj:liminf}
Let $X$, $\Phi$, $f$, and $x$ be as in Conjecture~\ref{conj:limsup}.
If $X$ is irreducible and $\cO_\Phi(x)$ is Zariski dense in $X$, then
\[
\liminf_{n\to\infty} \frac{h(f(\Phi^n(x)))}{\log n}>0.
\]
\end{conjecture}

Generalizing our method of proof of Theorem \ref{thm:limsup} via a more involved technique introduced in Section~\ref{sec:uniform-limsup}, we obtain a uniform version of the above $\limsup$ height gap result for any subset $T\subseteq \NN$ of positive density. See Definition~\ref{def:density} for the notion of upper asymptotic density.

\begin{theorem}[Uniform $\limsup$ Height Gaps]
\label{thm:uniform-limsup}
Let $X$, $\Phi$, $f$, and $x$ be as in Conjecture~\ref{conj:limsup}.
Then either $f(\OO_\Phi(x))$ is finite,
or there exists an $\epsilon>0$ such that for any subset $T\subseteq \NN$ of positive density, we have
\[
\limsup_{n\in T} \frac{h(f(\Phi^n(x)))}{\log n} > \epsilon.
\]
\end{theorem}

The significance of our above uniform bound is that it implies the $\liminf$ Height Gap Conjecture~\ref{conj:liminf} away from a set of density zero.

\begin{theorem}[Weak $\liminf$ Height Gaps]
\label{thm:weak-liminf}
Let $X$, $\Phi$, $f$, and $x$ be as in Conjecture~\ref{conj:limsup}.\footnote{Note that Conjecture \ref{conj:liminf} is stated only for $X$ is irreducible and $\cO_\Phi(x)$ Zariski dense which is necessary as shown in an example of \cite{BGS}, however our result holds without these hypotheses.}
If $f(\cO_\Phi(x))$ is infinite, then there is a constant $C>0$ and a set $S\subset \NN$ of upper asymptotic density zero such that
\[
h(f(\Phi^n(x))) > C \log n
\]
whenever $n\not\in S$, or equivalently,
\[
\liminf_{n\in \NN\setminus S} \frac{h(f(\Phi^n(x)))}{\log n} > 0.
\]
\end{theorem}

As an application of Theorem \ref{thm:limsup}, we obtain a simple proof of the univariate version of a result of Bell--Nguyen--Zannier \cite{BNZ} which, in turn, generalized results of van der Poorten--Shparlinski \cite{vdPS} with the aid of \cite{BC17}.

We recall that a power series $F(z)\in \Qbar[[z]]$ is {\it $D$-finite}, if it is the solution of a non-trivial homogeneous linear differential equation with coefficients in the rational function field $\Qbar(z)$; this is equivalent to saying that the coefficients of $F(z)$ satisfy certain linear recurrence relations with polynomial coefficients (see \cite[Theorem~1.5]{Stanley80}).

\begin{theorem}[Height gaps for $D$-finite power series]
\label{thm:D-finite-gap}
If $\sum_{n\geq0} a_n z^n\in \Qbar[[z]]$ is $D$-finite and $\displaystyle \limsup_{n\to\infty} \frac{h(a_n)}{\log n} = 0$, then the sequence $(a_n)_{n\in \NN}$ is eventually periodic.
\end{theorem}

As an application of Theorem \ref{thm:weak-liminf}, we show that the Dynamical Mordell--Lang Conjecture holds away from a set of density zero. We note that this result was obtained in \cite[Corollary~1.5]{DML-noetherian} using the upper Banach density function.

\begin{theorem}[Weak Dynamical Mordell--Lang]
\label{thm:weak-DML}
Let $X$ be a quasi-projective variety, $\Phi\colon X\dra X$ a rational self-map, and $Y\subseteq X$ a subvariety of $X$, all defined over $\Qbar$.
If $x\in X_\Phi(\Qbar)$, then $\{n\in\NN : \Phi^n(x)\in Y\}$ is a union of finitely many arithmetic progressions along with a set of upper asymptotic density zero.
\end{theorem}

Lastly, we prove a natural generalization of Theorem~\ref{thm:uniform-limsup} for commuting rational self-maps.
Given $m$ commuting rational self-maps $\Phi_1,\ldots,\Phi_m$ of $X$ and $\bn \coloneqq (n_1,\ldots,n_m) \in \NN^m$, we denote by $\Phi^{\bn}$ the composite $\Phi_1^{n_1}\circ \cdots \circ \Phi_d^{n_m}$. Let $X_{\Phi_1,\ldots,\Phi_m}(\Qbar)$ denote the subset of points $x\in X(\Qbar)$ such that for every $\bn\in \NN^m$, the $\bn$-th iterate $\Phi^{\bn}(x)$ avoids the indeterminacy loci of all $\Phi_1, \ldots, \Phi_m$. For any $x \in X_{\Phi_1,\ldots,\Phi_m}(\Qbar)$, we let $\cO_{\Phi_1,\ldots,\Phi_m}(x)$ be the set of points of the form $\Phi^{\bn}(x)$. Similarly, for a rational function $f\colon X \dra \bP^1$, we let $X_{\Phi_1,\ldots,\Phi_m,f}(\Qbar)\subseteq X_{\Phi_1,\ldots,\Phi_m}(\Qbar)$ denote the subset of points $x$ with $I_f\cap\cO_{\Phi_1,\ldots,\Phi_m}(x)=\varnothing$. We endow $\NN^m$ with the $1$-norm $\norm{\bn} \coloneqq n_1+\cdots+n_m$.

\begin{theorem}[$\limsup$ Height Gaps for commuting rational self-maps]
\label{thm:multi-limsup}
Let $X$ be a quasi-projective variety, let $\Phi_1,\ldots,\Phi_m$ be $m$ commuting rational self-maps of $X$, and let $f\colon X\dra \bP^1$ be a non-constant rational function, all defined over $\Qbar$.
Then for any $x\in X_{\Phi_1,\ldots,\Phi_m,f}(\Qbar)$, either $f(\cO_{\Phi_1,\ldots,\Phi_m}(x))$ is finite, or
there exists an $\epsilon>0$ such that for any subset $T \subseteq \NN$ of positive upper asymptotic density, we have
\[
\limsup_{\norm{\bn} \in T} \frac{h(f(\Phi^\bn(x)))}{\log\norm{\bn}} > \epsilon. 
\]
\end{theorem}

Theorem~\ref{thm:multi-limsup} is shown by induction, where the key step is the base case $m=1$; note that this base case is precisely Theorem~\ref{thm:uniform-limsup}. In a similar manner to the proof of Theorem~\ref{thm:weak-liminf}, one can deduce a weak $\liminf$ height gap result for commuting rational maps. In Example~\ref{ex:multi-density}, we show that one cannot expect a version of Theorem \ref{thm:multi-limsup} to hold if the $\limsup$ over $\norm{\bn}\in T$ is replaced by a $\limsup$ over $\bn \in \bT$.

In \cite{Lipshitz89}, Lipshitz introduced and studied multivariate $D$-finite power series (see Definition~\ref{def:multi-D-finite}), which extended Stanley's pioneering work \cite{Stanley80} on univariate $D$-finite power series.
Recently, the first author, Nguyen, and Zannier proved a height gap result for the coefficients of multivariate $D$-finite power series; see \cite[Theorem~1.3(c)]{BNZ}. The reader may be curious to know whether it is possible to deduce their result from Theorem~\ref{thm:multi-limsup}, analogously to how we deduced the univariate $D$-finiteness result Theorem~\ref{thm:D-finite-gap} from Theorem~\ref{thm:limsup}. This appears to be a subtle issue: our proof of Theorem~\ref{thm:D-finite-gap} relies on the fact that for sufficiently large $n$, the coefficients of a univariate $D$-finite power series are of the form $f(\Phi^n(c))$ for certain choices of $X$, $\Phi$, $f$, and $c$; see \cite[Section~3.2.1]{DML-book}. In contrast, we construct in Example~\ref{ex:multi-coeff} a $D$-finite power series in two variables (in fact a rational function) whose coefficients never arise as $f(\Phi_1^{n_1}\circ\Phi_2^{n_2}(c))$ for any choices of $X$, $\Phi_1$, $\Phi_2$, $f$, and $c$.


\section*{Acknowledgments}
We thank Ken Davidson and Vern Paulsen for helpful conversations.


\section{The $\limsup$ Height Gap Conjecture: Proof of Theorem~\ref{thm:limsup}}
\label{sec:limsup-gap}

We start by stating Schanuel's Theorem, which plays a central role in the proofs of Theorems \ref{thm:limsup} and \ref{thm:uniform-limsup}. It can be regarded as a quantitative version of Northcott's theorem \cite{Northcott49}. Schanuel's Theorem has a conjectural extension to Fano varieties, known as Manin's conjecture, which has attracted a lot of attention recently (see the survey \cite{LT19} and references therein).
\begin{theorem}[{Schanuel \cite{Schanuel79}, cf.~\cite[11.10.5]{BG06}}]
\label{thm:Schanuel}
Let $\KK$ be a number field of degree $d$ and let $h(\cdot)$ denote the absolute logarithmic Weil height on $\bP^n_{\KK}$. Then we have
\[
\lim_{B\to \infty} \frac{\#\{P\in \bP^n_\KK : h(P)\le \log B \}}{B^{d(n+1)}} = C_{n,\KK} > 0,
\]
where the positive constant $C_{n,\KK}$ depends only on $n$ and $\KK$.
\end{theorem}

We shall prove Theorem~\ref{thm:limsup} via an application of Schanuel's Theorem \ref{thm:Schanuel} and the following lemma. Recall that a topological space $U$ is called {\it Noetherian} if the descending chain condition holds for closed subsets of $U$, i.e.~for every chain of closed sets $Z_1\supset Z_2\supset\dots$, there is some $m\geq1$ for which $Z_m=Z_n$ for all $n\geq m$.

\begin{lemma}
\label{l:truncated-orbit}
Let $X$ be a quasi-projective variety, let $\Phi\colon X\dra X$ be a rational self-map, and let $f\colon X\dra \bP^1$ be a rational function, all defined over $\Qbar$.
Then there exists a constant $\ell\in\NN$ with the following property: if $x,y\in X_{\Phi,f}(\Qbar)$ and $f(\Phi^n(x))=f(\Phi^n(y))$ for $0\leq n\leq \ell$, then $f(\Phi^n(x))=f(\Phi^n(y))$ for all $n\geq0$.
\end{lemma}
\begin{proof}
Let $U_n = X \setminus \bigcup_{j \leq n} (I_{\Phi^j}\cup I_{f\circ\Phi^j})$ and $U = \bigcap_n U_n$. By construction, the $\Qbar$-points of $U$ are precisely those on which $\Phi^n$ and $f\circ\Phi^n$ are well-defined for all $n\geq0$, i.e.~$U(\Qbar)=X_{\Phi,f}(\Qbar)$. We endow $U$ with the subspace topology inherited from $X$ thereby making it a Noetherian topological space. Since
\[
\begin{tikzcd}
U\times U \arrow[hook]{r} & U_n\times U_n\arrow{r}{(f\circ\Phi^n, \, f\circ\Phi^n)} &[2.5em] \PP^1\times\PP^1
\end{tikzcd}
\]
is continuous and the image of the diagonal map $\PP^1\to\PP^1\times\PP^1$ is closed, we see that
\[
Z_n \coloneqq \{(x,y)\in U\times U : f(\Phi^i(x))=f(\Phi^i(y))\textrm{\ for\ }i\leq n\}
\]
is a closed subset of $U\times U$. As $U\times U$ is Noetherian, there exists an $\ell\in \NN$ such that $Z_n=Z_\ell$ for all $n\geq\ell$.
\end{proof}

\begin{proof}[Proof of Theorem \ref{thm:limsup}]
Let $x\in X_{\Phi,f}(\Qbar)$. 
Without loss of generality, we may assume that $X$, $\Phi$, $f$, and $x$ are defined over a fixed number field $\KK$. Suppose that
\[
\limsup_{n\to\infty} \frac{h(f(\Phi^n(x)))}{\log n}=0,
\]
i.e., $h(f(\Phi^n(x)))=o(\log n)$. We will show that $f(\mathcal{O}_\Phi(x))$ is finite.

Letting $\ell$ be as in Lemma \ref{l:truncated-orbit}, we define
\[
y_i \coloneqq (f(\Phi^{i}(x)), f(\Phi^{i+1}(x)), \dots, f(\Phi^{i+\ell}(x)))\in(\PP^1)^{\ell+1}(\KK)
\]
for $i\geq0$, and let $S = \{y_i : i\geq0\}$.
Via the Segre embedding, we may view $S\subseteq\PP^{2^{\ell+1}-1}(\KK)$.
Then
\[
h(y_i)=\sum_{j=0}^\ell h(f(\Phi^{i+j}(x)))=o(\log i).
\]

Next, choose $0<\epsilon<([\KK:\QQ]2^{\ell+1})^{-1}$. Then there exists $N_0\in\ZZ^+$ such that for all $i\geq N_0$, we have $h(y_i)<\epsilon\log i$. So, for all $n\geq N_0$,
\[
\#\{y_{N_0},y_{N_0+1},\dots,y_n\} \leq \# \big\{ z\in\PP^{2^{\ell+1}-1}(\KK) : h(z) \leq \log n^\epsilon \big\} = O(n^{\epsilon[\KK:\QQ]2^{\ell+1}}),
\]
where the equality comes from applying Schanuel's Theorem \ref{thm:Schanuel}.
Choosing $n$ sufficiently large, we find
\[
\#\{y_{N_0},y_{N_0+1},\ldots,y_n\}<n-N_0.
\]
In particular, there exist $i<j$ for which $y_i=y_j$. Thus, $f(\Phi^n(\Phi^i(x)))=f(\Phi^n(\Phi^j(x)))$ for all $0\leq n\leq \ell$, and so Lemma \ref{l:truncated-orbit} implies $f(\Phi^{n+i}(x))=f(\Phi^{n+j}(x))$ for all $n\geq0$. It follows that $f(\Phi^n(x))$ is eventually periodic with period dividing $j-i$. Hence, $f(\mathcal{O}_\Phi(x))$ is finite.
\end{proof}


\section{Uniform $\limsup$ Height Gap Result: Proof of Theorem~\ref{thm:uniform-limsup}}
\label{sec:uniform-limsup}

The main goal of this section is to prove Theorem~\ref{thm:uniform-limsup} which is the strengthening of Theorem~\ref{thm:limsup}.

\subsection{Preliminary results on sets of positive density}

\begin{definition}
\label{def:density}
Let $A$ be a subset of $\ZZ^+$.
The {\it upper asymptotic (or natural) density} $\den(A)$ of $A$ is defined by
\[
\den(A) \coloneqq \limsup_{m\to \infty} \frac{|A \cap [1, m]|}{m}.
\]
We frequently refer to $\den(A)$ simply as the {\it density} of $A$.
\end{definition}

\begin{remark}
\label{rmk:density}
It is easy to see that the density $\den(A)$ of any $A\subseteq \ZZ^+$ is {\it right translation invariant}, i.e., $\den(A+i) = \den(A)$ for any $i\in \NN$, where $A+i \coloneqq \{ a+i : a\in A\}$. Consequently, can extend the definition of density to any $A\subseteq\ZZ$ that is bounded from below.
\end{remark}

\begin{remark}
\label{rmk:extraction}
Let $T \subseteq \NN$ have positive density and let $L\ge 1$. By the subadditivity of natural density, there exists some $a\in \{0,1,\ldots ,L-1\}$ such that $T\cap (a+L\mathbb{N})$ has positive density.
\end{remark}

\begin{definition}
\label{def:shift-set}
Given $T \subseteq \NN$, the \emph{shift set} of $T$ is defined to be
\[
\Sigma(T)=
\{ i \in \NN : \den(T\cap (T+i))>0 \}.
\]
\end{definition}

Our goal in this subsection is to prove that if $T$ has positive density, then $\Sigma(T)$ does as well. We prove this after a preliminary lemma.

\begin{lemma}
\label{l:positive-density-shifts}
Let $T\subseteq \NN$ and $N \in \ZZ^+$ satisfying $\den(T)>\frac{1}{N}$.
Then for any finite subset $F\subseteq\NN$ with $|F|\ge N$, there exist $j,k\in F$ with $j>k$ such that $\den((T+(j-k))\cap T)>0$.
\end{lemma}
\begin{proof}
For ease of notation, we let $T_i=T+i$ for any $i\in \NN$. By definition of the density function, there is a sequence $0<m_1<m_2<\cdots<m_n<\cdots$ and intervals $I_n = [0, m_n] \subset \NN$ such that
\[
\lim_{n\to \infty} \frac{|T \cap I_n|}{|I_n|} = \den(T).
\]
For each $i\in\NN$, we have $|T\cap I_n| - i \le |T\cap (I_n-i)| \le |T\cap I_n|$, and so $\displaystyle \lim_{n\to \infty} \frac{|T \cap (I_n-i)|}{|I_n|} = \den(T)$. In particular, this holds for each $i\in F$.

Fix an $\epsilon>0$ with $\frac{1}{N}+\epsilon < \den(T)$. Then for $n$ sufficiently large,
\[
\frac{|T_i \cap I_n|}{|I_n|}=\frac{|T \cap (I_n-i)|}{|I_n|} > \frac{1}{N}+\epsilon
\]
for all $i\in F$. 
Now, suppose to the contrary that $\den(T_j \cap T_k) = \den(T_{j-k} \cap T) = 0$ for all distinct $j, k\in F$ with $j>k$. It follows that for $n$ sufficiently large,
\[
|T_j \cap T_k \cap I_n| < \frac{2 |I_n|}{|F|} \epsilon
\]
for all distinct $j, k\in F$.
Clearly,
\[
|I_n| \ge \big| I_n \cap \bigcup_{i\in F} T_i \big| = \big| \bigcup_{i\in F} (T_i \cap I_n) \big|.
\]
However, the inclusion-exclusion principle asserts that
\begin{align*}
\big| \bigcup_{i\in F} (T_i \cap I_n) \big| &\ge \sum_{i\in F} |T_i \cap I_n| - \sum_{\substack{{j, k\in F} \\ j>k}} |T_j \cap T_k \cap I_n| \\
&> |F| \big(\frac{1}{N}+\epsilon \big) |I_n| - {|F| \choose 2} \frac{2 |I_n|}{|F|} \epsilon \\
&= \big(\frac{|F|}{N}+\epsilon \big) |I_n| \ge (1+\epsilon)|I_n|,
\end{align*}
which yields a contradiction and hence Lemma~\ref{l:positive-density-shifts} follows.
\end{proof}

The following result is strengthening of \cite[Lemma~2.1]{DML-noetherian}.

\begin{proposition}
\label{prop:positive-density-shifts}
If $T\subseteq\NN$ satisfies $\den(T)>0$, then $\den(\Sigma(T))>0$.
\end{proposition}
\begin{proof}
Choose a positive integer $N$ satisfying $\den(T)>\frac{1}{N}$, and let $T_i$ denote $T+i$ for any $i \in \NN$. 
If $\Sigma(T)=\NN$, then there is nothing to prove.
So we may suppose there is some $i\in \NN$ such that $\den(T \cap T_i)=0$.
Consider the set $\mathcal{S}$ of those finite subsets $F\subseteq \NN$ such that $\den(T_{j-k}\cap T)=0$ for all $j,k\in F$ with $j>k$.
Clearly, $\mathcal{S} \neq \emptyset$ as $\{1, i+1\} \in \mathcal{S}$.
Moreover, by Lemma~\ref{l:positive-density-shifts}, we know $|F|<N$ for any $F\in \mathcal{S}$.

Let $\emptyset \neq F_{\max} \subseteq \NN$ be any maximal element of $\mathcal{S}$ (with respect to inclusion of sets),
and let $M$ be the maximum element of $F_{\max}$.
Then by our definition of $F_{\max}$, for any integer $n>M$, there exists some $k_n\in F_{\max}$ satisfying
\[
\den(T_{n-k_n}\cap T)>0, \text{ i.e., } n-k_n\in \Sigma(T).
\]
Since $0\leq k_n\leq M$, we see $n-M\leq n-k_n \le n$.
In particular, for every $c\geq2$, we have
\[
iM-k_{cM}\in \Sigma(T) \cap [(c-1)M, cM].
\]
It thus follows from the definition of density that
$\displaystyle \den(\Sigma(T)) \geq \lim_{c\to \infty} \frac{c-1}{cM} = \frac{1}{M}$.
\end{proof}

\begin{remark}
\label{rmk:R-of-T}
Using a similar argument, one can obtain an analogue of Proposition~\ref{prop:positive-density-shifts} where one replaces $\den$ by upper Banach density.
\end{remark}

\subsection{Stable non-periodic dimension}

Given a Noetherian topological space $U$ of finite Krull dimension and continuous map $\Phi\colon U\to U$, a subset $Y\subseteq U$ is {\it periodic with respect to $\Phi$} if $\Phi^n(Y) \subseteq Y$ for some positive integer $n$; we frequently say $Y$ is $\Phi$-{\it periodic} or simply {\it periodic} if $\Phi$ is understood from context.
It is well known that 
If $Z\subseteq U$ is a closed subset, let $Z_1, \ldots, Z_r$ denote its irreducible components and consider the set
\[
\mathcal{S} = \left\{ \bigcup_{i\in I} Z_i : I\subseteq \{1,\dots, r\} \right\}.
\]
Notice that if $Y_1, Y_2\in \mathcal{S}$ are periodic with respect to $\Phi$, then so is $Y_1\cup Y_2$. In particular, there is a unique maximal $\Phi$-periodic element of $\mathcal{S}$ which we denote by $P_{\Phi}(Z)$. Notice that $P_{\Phi}(Z)$ contains all periodic irreducible components of $Z$, but it is possible for $P_{\Phi}(Z)$ to also contain some non-periodic irreducible components of $Z$ as well. We let $N_{\Phi}(Z)$ denote the union of the irreducible components of $Z$ that are not contained in $P_{\Phi}(Z)$.

For each $Z_i\subseteq N_{\Phi}(Z)$, the sequence $\dim \ol{\Phi^n(Z_i)}$ is weakly decreasing and converges to some $d_i\in \NN$ since $U$ has finite Krull dimension. Let
\[
\nu_i \coloneqq (d_i,\dim Z_i).
\]
We put a strict total order $\prec$ on $(\mathbb{N}\cup \{-\infty\})\times(\mathbb{N}\cup \{-\infty\})$ by declaring $(a,b)\prec (a',b')$ if $a<a'$, or if $a=a'$ and $b<b'$. The relations $\preceq$, $\succ$, and $\succeq$ are then defined in the natural way.

\begin{definition}
\label{def:ess-non-per-dim}
With notation as above, we define the {\it stable non-periodic dimension} $\nu(Z)$ of $Z$ to be the maximum $\nu_i$ with respect to $\prec$.
If $N_{\Phi}(Z)$ is empty, we define $\nu(Z) = (-\infty,-\infty)$.
\end{definition}

The following is the main technical result of this subsection.

\begin{proposition}
\label{prop:strictly-decreasing}
Let $U$ be a Noetherian topological space of finite Krull dimension and $\Phi\colon U\to U$ a continuous map. Suppose that $T\subseteq\NN$ has positive density and $Z\subseteq U$ is a closed subset with $N_{\Phi}(Z)\neq\varnothing$. Then there exist infinitely many $j\in \Sigma(T)$ with $\nu(Z) \succ \nu(Z\cap \Phi^{-j}(Z))$.
\end{proposition}
\begin{proof}
Let $Z_1, \ldots, Z_r$ be the irreducible components of $Z$. After relabelling, we may assume
\[
N_{\Phi}(Z)=Z_1\cup\dots\cup Z_s\quad\quad\text{and}\quad\quad P_{\Phi}(Z)=Z_{s+1}\cup\dots\cup Z_r.
\]
Let $L \in \ZZ^+$ such that $\Phi^L(P_{\Phi}(Z))\subseteq P_{\Phi}(Z)$. By Remark~\ref{rmk:extraction}, there is some $a\in \{0,1,\ldots,L-1\}$ such that $T\cap (a+L\mathbb{N})$ has positive density. Replacing $T$ by $T\cap (a+L\mathbb{N})$, we can assume that all elements of $\Sigma(T)$ are multiples of $L$.

Fix $m$ sufficiently large so that
\[
\dim\ol{\Phi^m(Z_i)}=\lim_{n\to\infty}\dim\ol{\Phi^n(Z_i)}
\]
for $i\leq s$, and let $\nu(Z)=(d,e)$. After relabeling, we may assume that there exist $1\leq \ell\leq t\leq s$ such that:
\begin{enumerate}
\item $\dim\ol{\Phi^m(Z_i)}=d$ and $\dim Z_i=e$ for $i\le \ell$,
\item $\dim\ol{\Phi^m(Z_i)}=d$ and $\dim Z_i<e$ for $\ell<i\le t$,
\item $\dim\ol{\Phi^m(Z_i)}<d$ for $t<i\le s$.
\end{enumerate}

We first claim that for every $j\in\Sigma(T)$, we have
\begin{equation}
\label{eq:P-Phi}
P_{\Phi}(Z\cap \Phi^{-j}(Z)) \supseteq P_{\Phi}(Z).
\end{equation}
To see this, first note that since $j$ is a multiple of $L$, we have $Z\cap \Phi^{-j}(Z)\supseteq P_{\Phi}(Z)$. So, it remains to show that every irreducible component $Z_i$ contained in $P_{\Phi}(Z)$ is also an irreducible component of $Z\cap \Phi^{-j}(Z)$. Since $Z_i$ is irreducible, it is contained in some irreducible component $Z'_i$ of $Z\cap \Phi^{-j}(Z)$. Then $Z_i \subseteq Z'_i \subseteq Z\cap \Phi^{-j}(Z) \subseteq Z$. As $Z_i$ is already an irreducible component of $Z$ and $Z'_i$ is irreducible, it follows that $Z_i = Z'_i$ is an irreducible component of $Z\cap \Phi^{-j}(Z)$.

By \eqref{eq:P-Phi}, we necessarily have $\nu(Z\cap \Phi^{-j}(Z)) \preceq \nu(Z)$. Suppose that
\begin{equation}
\label{eq:assumption}
\nu(Z\cap \Phi^{-j}(Z)) = \nu(Z)
\end{equation}
for every sufficiently large $j\in \Sigma(T)$. We shall derive a contradiction in the remainder of the proof.

We next claim that for every sufficiently large $j\in \Sigma(T)$, there is some $i\le \ell$ such that
\begin{equation}
\label{eq:dim}
Z_i\subseteq \Phi^{-j}(Z).
\end{equation}
If $j\in\Sigma(T)$ is sufficiently large, then by \eqref{eq:assumption}, there is an irreducible component $C$ of $Z\cap \Phi^{-j}(Z)$ not contained in $P_{\Phi}(Z\cap \Phi^{-j}(Z))$ such that $\dim C = e$ and $\dim \ol{\Phi^n(C)} \ge d$ for all $n\ge 0$. We have $C \subseteq Z_i\cap \Phi^{-j}(Z)$ for some $1\le i\le r$.
By \eqref{eq:P-Phi}, we see $C\not\subseteq P_{\Phi}(Z)$, and so $Z_i$ is not an irreducible component contained in $P_{\Phi}(Z)$, i.e., $i\le s$. Next observe that
\[
d\le \dim \ol{\Phi^m(C)} \le \dim\ol{\Phi^m(Z_i\cap \Phi^{-j}(Z))} \le \dim\ol{\Phi^m(Z_i)},
\]
and so $i\leq t$. Moreover, since $C \subseteq Z_i\cap \Phi^{-j}(Z)\subseteq Z_i$, we see $\dim Z_i\geq e$ and hence $i\leq\ell$. By dimension contraints, $C=Z_i$ which implies \eqref{eq:dim}.

Since Proposition \ref{prop:positive-density-shifts} shows that $\Sigma(T)$ has positive density, by the subadditivity of natural density, there exists a fixed $i\in \{1, \ldots, \ell \}$ and a positive density subset $\Sigma_i\subseteq\Sigma(T)$ such that equation~\eqref{eq:dim} holds for all $j\in\Sigma_i$.
Further refining, there exists $k\in\{1,\ldots, r\}$ and a positive density subset $\Sigma_{i,k}\subseteq\Sigma_i$ such that
\[
Z_i\subseteq \Phi^{-j}(Z_k)
\]
for all $j\in \Sigma_{i,k}$.

We next show that $k\leq s$. If this were not the case, then $Z_k\subseteq P_{\Phi}(Z)$ and so $\Phi^j(Z_i)\subseteq Z_k\subseteq P_{\Phi}(Z)$.
In particular, since $j$ is a multiple of $L$, we have $\Phi^j(Z_i \cup P_{\Phi}(Z))\subseteq P_{\Phi}(Z)\subseteq Z_i \cup P_{\Phi}(Z)$. By maximality of $P_{\Phi}(Z)$, it follows that $P_{\Phi}(Z)=Z_i \cup P_{\Phi}(Z)$, and hence $Z_i\subseteq P_{\Phi}(Z)$, a contradiction.

Since $\Sigma_{i,k}$ is infinite, there exist $a,b\in \Sigma_{i,k}$ with $b-a,a>m$.
We write $b=a+Lc$ with $c>0$. Since $Z_i\subseteq\Phi^{-a}(Z_k)$, we have $\Phi^{a+Lc}(Z_i)= \Phi^{Lc}(\Phi^a(Z_i))\subseteq \Phi^{Lc}(Z_k)$, and hence $\ol{\Phi^{a+Lc}(Z_i)}\subseteq \ol{\Phi^{Lc}(Z_k)}$.
As $a+Lc,Lc>m$, we see $\dim\ol{\Phi^{Lc}(Z_k)}\leq d=\dim\ol{\Phi^{a+Lc}(Z_i)}$.
Then by irreducibility of $Z_k$, we have $\overline{\Phi^{a+Lc}(Z_i)} = \overline{\Phi^{Lc}(Z_k)}$.
On the other hand, $b\in\Sigma_{i,k}$, so $\Phi^{a+Lc}(Z_i)\subseteq Z_k$, which implies 
\[
\Phi^{Lc}(Z_k)\subseteq\overline{\Phi^{Lc}(Z_k)}=\overline{\Phi^{a+Lc}(Z_i)}\subseteq Z_k.
\]
So, $Z_k$ is periodic and hence contained in $P_{\Phi}(Z)$, contradicting the fact that $k\le s$.
\end{proof}

\begin{lemma}
\label{lem:prec}
Let $U$ be a non-empty Noetherian topological space of Krull dimension $d$. Suppose that
\[
Z_0\supseteq Z_1\supseteq Z_2\supseteq \cdots \supseteq Z_m
\]
is a descending chain of non-periodic closed subsets of $U$ such that $\nu(Z_0)\succ \nu(Z_1)\succ \cdots \succ \nu(Z_m)$.
Then $m < (d+1)^2$.
\end{lemma}
\begin{proof}
We necessarily have $\nu(Z_0)\preceq (d,d)$. Write $\nu(Z_i)=(d_i,e_i)$.  Then by definition of $\prec$, we have $d\ge d_0\ge d_1\ge \cdots \ge d_m$.  For $s\in \{0,1,\ldots ,d\}$, let $A_s=\{i : d_i=s\}$.  Then $A_s$ is an interval.  Notice that if $A_s=\{j,j+1,\ldots ,j+\ell\}$, then since $d_j = \cdots = d_{j+\ell} = s$, we must have $e_j>e_{j+1}>\cdots > e_{j+\ell}$.  Since $e_j\le d$, we see that $\ell\le d$.  
Hence $|A_s| \le d+1$ for each $s\in \{0,1,2,\ldots ,d\}$.  Then since $\{0,1,2,\ldots ,m\}=A_0\cup A_1\cup \cdots \cup A_d$, we see that $m+1\le (d+1)^2$, as required.
\end{proof}

\subsection{Finishing the proof of Theorem~\ref{thm:uniform-limsup}}
 
\begin{proof}[Proof of Theorem~\ref{thm:uniform-limsup}]

We divide the proof into several steps.

\begin{step}
\label{step:set-up}
We shall start with the following set-up.
Let
\[
U = X \setminus \bigcup_{n\in \NN} (I_{\Phi^n}\cup I_{f\circ\Phi^n}),
\]
which may not be open. We endow $U$ with the subspace topology, thereby making it a Noetherian topological space. Clearly, $U(\Qbar)=X_{\Phi,f}(\Qbar)$. If $U$ does not have any $\Qbar$-points, then the theorem is vacuously true, so we assume that there is an $x\in U(\Qbar)$ such that $f(\cO_\Phi(x))$ is infinite. We may also assume that $X$, $\Phi$, $f$ and $x$ are all defined over a fixed number field $\KK$. By construction, $\Phi|_U$ is a regular self-map of $U$ and $f|_U \colon U \to \bP^1$ is regular; by abuse of notation, we denote these restriction maps by $\Phi$ and $f$, respectively. Finally, replacing $U$ by the Zariski closure of the orbit $\cO_\Phi(x)$ in $U$, we may assume $\cO_\Phi(x)$ is Zariski dense in $U$; see also the end of the introduction of \cite{BGS}.

So, from now on, we may assume:
\begin{enumerate}
\item $U\subseteq X$ is a Noetherian topological space;
\item $\Phi$ is regular on $U$ and $\Phi(U)\subseteq U$;
\item $f \colon U \to \PP^1$ is regular;
\item $x\in U(\Qbar)$ whose orbit $\cO_\Phi(x)$ is Zariski dense in $U$;
\item $f(\cO_\Phi(x))$ is infinite.
\end{enumerate}
\end{step}

\begin{step}
\label{step:construct-Zm}
Let $T \subseteq \NN$ be a subset of positive density, $d = \dim (U \times U)$, and
\[
Z_0=\{(u,v) \in U\times U : f(u) = f(v)\}.
\]
Note that $Z_0$ is a closed subset of $U\times U$ since it is the inverse image of the diagonal $\Delta_{\bP^1} \subseteq \bP^1 \times \bP^1$ under the product map $(f, f) \colon U \times U \to \bP^1 \times \bP^1$.
Applying Proposition~\ref{prop:strictly-decreasing} to $Z_0 \subset U\times U$, the product map $(\Phi, \Phi)$, and $T$, we see that there is some $i_0\in \Sigma(T)$ such that $T_1 \coloneqq T\cap (T+i_0)$ has positive density and $Z_1 \coloneqq Z_0\cap (\Phi,\Phi)^{-i_0}(Z_0)$ satisfies $\nu(Z_1) \prec \nu(Z_0)$. If $Z_1 = P_{(\Phi, \Phi)}(Z_1)$ is periodic under $(\Phi, \Phi)$, then let $m=1$. Otherwise, applying Proposition~\ref{prop:strictly-decreasing} to $Z_1$ yields an element $i_1\in \Sigma(T_1)\subseteq\Sigma(T)$ with $i_0<i_1$ such that $T_2 \coloneqq T_1\cap (T_1+i_1)$ has positive density and $Z_2 \coloneqq Z_1\cap (\Phi,\Phi)^{-i_1}(Z_1)$ satisfies $\nu(Z_2) \prec \nu(Z_1)$. Proceeding in this manner, we obtain a sequence of integers
\[
0<i_0<i_1<\dots<i_m
\]
and a descending chain of closed subsets $Z_0\supseteq Z_1\supseteq Z_2\supseteq \cdots \supseteq Z_m$ such that $\nu(Z_i) \succ \nu(Z_{i+1})$ and $Z_m = P_{(\Phi, \Phi)}(Z_m)$, i.e., $Z_m$ is periodic. Furthermore, by construction, if
\[
S \coloneqq \left\{ \sum_{i\in I} i : I \subseteq \{ i_0, \dots, i_m\} \right\}
\]
then
\[
T' \coloneqq \bigcap_{s \in S} (T+s) \subseteq T
\]
has positive density,
\[
Z_m = \bigcap_{s\in S}(\Phi,\Phi)^{-s}(Z_0)
\]
and there is some $L\in \ZZ^+$ such that
\[
(\Phi, \Phi)^L(Z_m) \subseteq Z_m.
\]
Lastly, Lemma~\ref{lem:prec} implies
\[
|S| \le 2^{m+1} \le 2^{(d+1)^2}.
\]
Notice that $Z_m \neq \emptyset$ since the diagonal $\Delta_U\subseteq U\times U$ is contained in $Z_0$ and $(\Phi,\Phi)^{-n}(\Delta_U)\supseteq \Delta_U$ for every $n\ge 0$. 
\end{step}

\begin{step}
\label{step:Schanuel-argument}

By Schanuel's Theorem \ref{thm:Schanuel}, there exists a positive real number $\kappa>0$ depending only on the number field $\KK$ such that for all sufficiently large $B$, we have
\begin{equation}
\label{eqn:Schanuel}
\# \big \{ y\in \bP^1(\KK) : h(y) \le \log B \big \} \le B^{\kappa}.
\end{equation}
Choose an $\epsilon$ independent of $T$ such that $0<\epsilon<(2^{(d+1)^2+1}\kappa)^{-1}$.
We shall prove that this choice of $\epsilon$ satisfies the conclusion of Theorem~\ref{thm:uniform-limsup}.
Suppose to the contrary that
\[
\limsup_{n\in T} \frac{h(f(\Phi^n(x)))}{\log n} \le \epsilon.
\]
In particular, there is a positive integer $N_0$ such that
$h(f(\Phi^n(x))) \le 2\epsilon \log n$
for all $n\in T$ with $n\geq N_0$.

First, by equation (\ref{eqn:Schanuel}), we have
\begin{equation}
\label{eq:Schanuel-on-orbit}
\# \left \{f(\Phi^n(x)) \in \PP^1(\KK) : n\in T,\; N_0\leq n\leq N \right \} \le N^{2\epsilon\kappa}
\end{equation}
for $N$ sufficiently large. Let
\[
T'' = \bigcap_{s \in S} (T-s) \subseteq T.
\]
Since $\den(T')>0$ and $T''=T'-(i_0+\dots+i_m)$, we see $\den(T'')>0$. By construction, for any $j\in T''$, we have $j+s\in T$ for every $s\in S$.
In particular, equation \eqref{eq:Schanuel-on-orbit} implies
\begin{equation}
\label{eq:P1-to-the-S}
\begin{aligned}
\# & \left\{ \left( f(\Phi^{j+s}(x)) \right)_{s\in S} \in \bP^1(\KK)^{|S|} : j\in T'',\; N_0\le j\le N - (i_0+\dots+i_m) \right\} \\
\le \ & \prod_{s\in S} \# \left\{ f(\Phi^{j+s}(x)) \in \bP^1(\KK) : j\in T'',\; N_0-s \le j \le N-s \right\} \\
\le \ & \prod_{s\in S} \# \left\{ f(\Phi^{j+s}(x)) \in \bP^1(\KK) : j+s\in T,\; N_0 \le j+s \le N \right\} \\
\le \ & N^{2\epsilon\kappa\cdot |S|}.
\end{aligned}
\end{equation}
\end{step}

\begin{step}
\label{step:lower-bnd}
Let $L$ be as in Step~\ref{step:construct-Zm}. Since $\den(T''\cap[N_0,\infty)) = \den(T'')>0$, by subadditivity of natural density, there exists an integer $a\in [0, L)$ such that $\{ j\in T'' : j\ge N_0, \; j \equiv a \pmod{L} \}$ has positive density. Then by the definition of natural density (see Definition~\ref{def:density}), there exist a subsequence $(n_\ell)_{\ell\in \ZZ^+}$ of positive integers and a positive real number $\delta>0$, such that
\[
\# \big\{ j\in T'' : N_0 \le j \le n_\ell, \; j \equiv a \mathrm{\ (mod\ } L) \big\} \ge \delta n_\ell
\]
for sufficiently large $\ell$. Now, replacing $\delta$ by a smaller positive number if necessary, we can further assume that
\begin{equation}
\label{eq:arithmeic-progression-pos-density}
\# \big \{j\in T'' : N_0\le j\le n_\ell - (i_0+\cdots+i_m), \; j \equiv a \mathrm{\ (mod\ } L) \big \} \ge \delta n_\ell,
\end{equation}
for sufficiently large $\ell$.

Recall from Step~\ref{step:construct-Zm} that $|S| \le 2^{m+1} \le 2^{(d+1)^2}$ and hence $2\epsilon\kappa\cdot |S| < 1$.
Therefore, we can choose $\ell$ large enough such that $\delta n_\ell > n_\ell^{2\epsilon\kappa\cdot |S|}$. 
Combining equation \eqref{eq:arithmeic-progression-pos-density} with \eqref{eq:P1-to-the-S} where $N=n_\ell$, a direct counting argument yields that there exist positive integers $i,j\in T''$ with $i<j$ such that
\[
f(\Phi^{i+s}(x))=f(\Phi^{j+s}(x)) \text{ for all } s\in S \quad\text{and}\quad i \equiv j \mathrm{\ (mod\ }L).
\]

Then by definition, $(\Phi^i(x), \Phi^j(x)) \in Z_m$ (see Step~\ref{step:construct-Zm} for the construction and properties of $Z_m$). Since $(\Phi, \Phi)^L(Z_m) \subseteq Z_m$, we have $(\Phi^{kL+i}(x), \Phi^{kL+j}(x))\in Z_m$ for every $k\in \NN$. As $Z_m \subseteq Z_0$, the definition of $Z_0$ yields 
\[
f(\Phi^{kL+i}(x)) = f(\Phi^{kL+j}(x)).
\]
It thus follows from the fact that $i \equiv j \pmod L$ that the sequence $\{ f(\Phi^{kL}(\Phi^i(x))) : k\in \NN\}$ is periodic.
In particular, the orbit $\cO_{\Phi^{L}}(\Phi^i(x))$ of $\Phi^i(x)$ under $\Phi^L$ is contained in finitely many fibers $F_1, \ldots, F_s$ of $f$.
Note that
\[
\cO_\Phi(\Phi^i(x)) = \bigcup_{t=0}^{L-1} \cO_{\Phi^L}(\Phi^{i+t}(x)) = \bigcup_{t=0}^{L-1} \Phi^t(\cO_{\Phi^L}(\Phi^{i}(x))) \subseteq \bigcup_{t=0}^{L-1} \Phi^t(F_1 \cup \cdots \cup F_s).
\]
Therefore, the full forward orbit $\cO_\Phi(x)$ is contained in some proper closed subset of $U$, which contradicts our assumption in Step~\ref{step:set-up} that $\cO_\Phi(x)$ is Zariski dense in $U$.
We thus complete the proof of Theorem~\ref{thm:uniform-limsup}.
\qedhere
\end{step}
\end{proof}


\section{Applications of Theorems \ref{thm:limsup} and \ref{thm:uniform-limsup}}
\label{sec:apps}

\subsection{Weak $\liminf$ Height Gaps}

Theorem~\ref{thm:weak-liminf}, which asserts that Conjecture~\ref{conj:liminf} holds away from a set of density zero, is an immediate consequence of Theorem~\ref{thm:uniform-limsup}.

\begin{proof}[Proof of Theorem \ref{thm:weak-liminf}]
Suppose that $f(\OO_\Phi(x))$ is infinite.
Let $\epsilon>0$ be the positive real number as in Theorem~\ref{thm:uniform-limsup} and
\[
S \coloneqq \left\{ n\in\NN : \frac{h(f(\Phi^n(x)))}{\log n}\leq\frac{\epsilon}{2} \right\}.
\]
Then by Theorem~\ref{thm:uniform-limsup}, $S$ has density zero, which concludes the proof.
\end{proof}

\subsection{Height gaps for $D$-finite power series}

We apply Theorem~\ref{thm:limsup} to obtain a simple proof of Theorem~\ref{thm:D-finite-gap} recovering the univariate case of \cite[Theorem~1.3(c)]{BNZ}.

\begin{proof}[Proof of Theorem~\ref{thm:D-finite-gap}]
If $\sum_{n\geq0} a_n z^n \in \Qbar[[z]]$ is a $D$-finite power series, then there is a rational self-map $\Phi\colon\mathbb{P}^d\dra \mathbb{P}^d$ for some $d\ge 2$, a point $c\in \mathbb{P}^d(\Qbar)$, and a rational map $f\colon \mathbb{P}^d\dra \mathbb{P}^1$ such that $a_n = f(\Phi^n(c))$ for $n\gg 0$;
see e.g., \cite[Section~3.2.1]{DML-book}.
So, Theorem~\ref{thm:limsup} immediately implies Theorem~\ref{thm:D-finite-gap}.
\end{proof}

\subsection{Weak Dynamical Mordell--Lang}

As mentioned before in the introduction, the $\liminf$ Height Gap Conjecture~\ref{conj:liminf} would imply the Dynamical Mordell--Lang Conjecture.
Similarly, we deduce Theorem~\ref{thm:weak-DML} as an application of Theorem~\ref{thm:weak-liminf}.

\begin{proof}[Proof of Theorem~\ref{thm:weak-DML}]
For any $n\in \NN$, we denote by $Z_{\ge n}$ the Zariski closure of $\{\Phi^i(x) : i\ge n\}$ in $X$. Since $X$ is a Noetherian topological space, there is some $m\in \NN$ such that $Z_{\ge n} = Z_{\ge m}$ for every $n\ge m$. Denote $Z_{\ge m}$ by $Z$ and $\Phi^m(x)$ by $x_1$. It then suffices to prove Theorem~\ref{thm:weak-DML} for $(Z, \Phi|_{Z}, x_1, Y\cap Z)$.

Let $Z_1, \ldots, Z_r$ denote the irreducible components of $Z$ and let $Y_i \coloneqq Y\cap Z_i$. Then $x_1 \in Z_i$ for some $i$. After relabeling, we may assume that $x_1 \in Z_1$. For each $i=2,\ldots,r$, choose an arbitrary $x_i \in \cO_\Phi(x_1) \cap Z_i$; the intersection is non-empty since $\cO_\Phi(x_1)$ is dense in $Z$ by definition. We claim that $\Phi|_Z$ cyclically permutes the irreducible components $Z_i$ of $Z$. To see this, first note that $\Phi|_Z$ is a dominant rational self-map of $Z$, so it permutes the $Z_i$. Suppose that $(Z_{i_1}, \ldots, Z_{i_s})$ is a cycle under $\Phi|_Z$ with $1\le s \le r$, and consider the forward orbit $\cO_\Phi(x_{i_1})$ of $x_{i_1} \in Z_{i_1}$. Clearly, the closure of $\cO_\Phi(x_{i_1})$ in $X$ is contained in the union of $Z_{i_1}, \ldots, Z_{i_s}$. On the other hand, $x_{i_1}=\Phi^n(x)$ for some $n\ge m$, and so the closure of $\cO_\Phi(x_{i_1})$ in $X$ is $Z_{\ge n}=Z$. It follows that $s=r$ and hence $\Phi|_Z$ is a cyclic permutation $(Z_{i_1}, \ldots, Z_{i_r})$.  Hence $\Phi^r(Z_i)\subseteq Z_i$ for each $i$. Moreover, after relabeling, we may assume that $\Phi(Z_i) \subseteq Z_{i+1}$ for $i=1,\ldots,r-1$ and $\Phi(Z_r) \subseteq Z_{1}$. So, for $i=2,\ldots,r$, our $x_i$ could be taken to be $\Phi^{i-1}(x_1)$. Therefore, by subadditivity of natural density, it suffices to show Theorem~\ref{thm:weak-DML} for $(Z_i, \Phi^r|_{Z_i}, x_i, Y_i)$ for each $i=1, \ldots, r$.

We claim further that for each $i$, the forward orbit $\cO_{\Phi^r}(x_i)$ of $x_i\in Z_i$ under $\Phi^r$ is dense in $Z_i$.
In fact, if we denote the irreducible decomposition of the closure of $\cO_{\Phi^r}(x_i)$ by $W_{i,1},\ldots,W_{i,r_i}$, then 
\[
\bigcup_{i=1}^r Z_i = Z = \overline{\cO_\Phi(x_1)} = \bigcup_{i=1}^r \overline{\cO_{\Phi^r}(x_i)} = \bigcup_{i=1}^r \bigcup_{j=1}^{r_i} W_{i,j}.
\]
Since $Z_i$ is irreducible, $Z_i \subseteq W_{k,j}$ for some $1\le k \le r$ and $1\le j \le r_k$.
However, we note that $W_{k,j} \subseteq \overline{\cO_{\Phi^r}(x_k)} \subseteq Z_k$.
As $Z_1,\ldots,Z_r$ are the irreducible components of $Z$, we must have $k=i$.
The claim $\overline{\cO_{\Phi^r}(x_i)} = Z_i$ thus follows.

We shall prove that either $\cO_{\Phi^r}(x_i)\subseteq Y_i$, or the set
\[
A_i \coloneqq \left\{ n\in \NN : \Phi^{rn}(x_i) \in Y_i \right\}
\]
has zero density, thereby proving Theorem~\ref{thm:weak-DML} for $(Z_i, \Phi^r|_{Z_i}, x_i, Y_i)$. If $Y_i=Z_i$ or $Y_i=\emptyset$, then the result is immediate. Thus we may assume, without loss of generality, that $Y_i$ is a non-empty proper subvariety of $Z_i$. We pick a non-constant morphism $f_i\colon Z_i \to \mathbb{P}^1$ such that $f_i(Y_i)=1$; one can accomplish this by choosing a non-constant rational function $F_i$ vanishing on $Y_i$ and then letting $f_i \coloneqq F_i + 1$. In particular, if $\Phi^{rn}(x_i)\in Y_i$, then $h(f_i(\Phi^{rn}(x_i)))=0$.
On the other hand, as $\cO_{\Phi^r}(x_i)$ is dense in $Z_i$, $f_i(\cO_{\Phi^r}(x_i))$ is necessarily infinite.
So, by Theorem~\ref{thm:weak-liminf}, there is a positive constant $C$ and a set $S\subset\NN$ of zero density such that for any $n\in \NN \setminus S$, the height of $f_i(\Phi^{rn}(x_i))$ is greater than $C\log n > 0$; in particular, for such an $n$, $\Phi^{rn}(x_i)\notin Y_i$.
It follows that $A_i \subseteq S$ has zero density, as required.
We hence complete the proof of Theorem~\ref{thm:weak-DML}.
\end{proof}


\section{Height gaps for commuting rational self-maps: Proof of Theorem \ref{thm:multi-limsup}}
\label{sec:multivariate-ht-gaps}

We begin this section by proving Theorem~\ref{thm:multi-limsup}.

\begin{proof}[Proof of Theorem~\ref{thm:multi-limsup}]
We fix a number field $\KK$ such that $X$, $\Phi_1,\ldots,\Phi_m$, $f$ and $x$ are defined over $\KK$.
We shall prove this theorem by induction on $m$.
When $m=1$, it reduces to Theorem~\ref{thm:uniform-limsup}.
Let us assume $m\ge 2$.
Suppose that Theorem~\ref{thm:multi-limsup} holds true for the semigroup $S_{m-1}$ generated by $\Phi_1,\ldots,\Phi_{m-1}$.
Assume that $f(\cO_{\Phi_1,\ldots,\Phi_m}(x))$ is infinite.
If $f(\cO_{\Phi_1,\ldots,\Phi_{m-1}}(x))$ is infinite, then by the induction hypothesis, we have
\[
\limsup_{\norm{\bn} \in T} \frac{h(f(\Phi^\bn(x)))}{\log \norm{\bn}} \ge
\limsup_{\substack{\norm{\bn} \in T \\ \bn=(n_1,\ldots,n_{m-1},0)\in \NN^m}} \frac{h(f(\Phi^\bn(x)))}{\log \norm{\bn}} > \epsilon_1,
\]
where $\epsilon_1 > 0$ depends only on $\KK$ and $X$.
Thus we may assume that $f(\cO_{\Phi_1,\ldots,\Phi_{m-1}}(x))$ is finite.

As in Step~\ref{step:set-up} of the proof of Theorem~\ref{thm:uniform-limsup}, there exists a Noetherian topological space $U$ such that $U(\Qbar) = X_{\Phi_1,\ldots,\Phi_m,f}(\Qbar)$, the $\Phi_i$ are regular self-maps of $U$, $f\colon U \to \bP^1$ is also regular, and $x\in U(\Qbar)$.
Now, let $Z$ denote the Zariski closure of the orbit $\cO_{\Phi_1,\ldots,\Phi_{m-1}}(x)$ of $x$ under the semigroup $S_{m-1} = \langle \Phi_1,\ldots,\Phi_{m-1} \rangle$ in $U$, and let $Z_1,\ldots,Z_r$ be the irreducible components of $Z$.
By assumption, $\cO_{\Phi_1,\ldots,\Phi_{m-1}}(x)$ is contained in finitely many fibers of $f \colon U \to \bP^1$, then so is its closure $Z$.
In particular, $f$ is constant on each $Z_i$.
Note that
\[
f(\cO_{\Phi_1,\ldots,\Phi_{m}}(x)) = \bigcup_{j\in \NN} f({\Phi_m^j(\cO_{\Phi_1,\ldots,\Phi_{m-1}}(x))}) \subseteq \bigcup_{i=1}^r \bigcup_{j\in \NN} f({\Phi_m^j(Z_i)}).
\]
There are two cases to consider.

Case 1. For every $i=1,\ldots,r$ and every $j \in \NN$, we have that the image of $\Phi_m^j(Z_i)$ under $f$ is finite; in particular, $f$ is constant on $\Phi_m^j(Z_i)$.
Then by assumption, there is some $i$ such that the sequence $\{f({\Phi_m^j(Z_i)}) : j \in \NN\}$ is infinite.
Let us pick some $z_i \in Z_i$ from the orbit $\cO_{\Phi_1,\ldots,\Phi_{m-1}}(x)$ of $x$, say $z_i = \Phi_1^{k_1}\circ \cdots \circ \Phi_{m-1}^{k_{m-1}}(x)$ for some fixed $k_1,\ldots,k_{m-1}$.
Then the result follows from the univariate case applying to $z_i$ and $\Phi_m$.
In fact,
\begin{align*}
\limsup_{\norm{\bn} \in T} \frac{h(f(\Phi^\bn(x)))}{\log \norm{\bn}} &\ge
\limsup_{\substack{\norm{\bn} \in T \\ \bn=(k_1,\ldots,k_{m-1},n_m)\in \NN^m}} \frac{h(f(\Phi^{\bn}(x)))}{\log \norm{\bn}} \\
&= \limsup_{n_m \in \, T-\sum k_{i}} \frac{h(f(\Phi_m^{n_m}(z_i)))}{\log n_m} \\
&> \epsilon_2,
\end{align*}
where $\epsilon_2 > 0$ from Theorem~\ref{thm:uniform-limsup} depends only on $\KK$ and $X$.

Case 2. There is some fixed $i_0$ with $1\le i_0 \le r$ and some fixed $j_0\ge 1$ such that the map induced by $f$ from the closure of $\Phi_m^{j_0}(Z_{i_0})$ to $\bP^1$ is dominant.
Note that
\[
\overline{{\Phi_m^{j_0}(\cO_{\Phi_1,\ldots,\Phi_{m-1}}(x))}} \supseteq {\Phi_m^{j_0}(\cup_{i=1}^r Z_i)} = \bigcup_{i=1}^r {\Phi_m^{j_0}(Z_i)}.
\]
So the restriction of $f$ to the closure of $\Phi_m^{j_0} \circ S_{m-1}(x) = \Phi_m^{j_0}(\cO_{\Phi_1,\ldots,\Phi_{m-1}}(x))$ in $U$ is dominant.
But $\Phi_m$ commutes with $S_{m-1}$ and so letting $x_{j_0}=\Phi_m^{j_0}(x)$, we see that the restriction of $f$ to the closure of $\cO_{\Phi_1,\ldots,\Phi_{m-1}}(x_{j_0})$ is dominant.
In particular, $f(\cO_{\Phi_1,\ldots,\Phi_{m-1}}(x_{j_0}))$ is infinite.
Now by the induction hypothesis, we also get
\begin{align*}
\limsup_{\norm{\bn} \in T} \frac{h(f(\Phi^\bn(x)))}{\log \norm{\bn}} &\ge
\limsup_{\substack{\norm{\bn} \in T \\ \bn=(n_1,\ldots,n_{m-1},j_0)\in \NN^m}} \frac{h(f(\Phi^\bn(x)))}{\log \norm{\bn}} \\
&= \limsup_{\substack{\norm{\bn} \in T-j_0 \\ \bn=(n_1,\ldots,n_{m-1})\in \NN^{m-1}}} \frac{h(f(\Phi^\bn(x_{j_0})))}{\log \norm{\bn}}\\
&> \epsilon_1.
\end{align*}
Set $\epsilon \coloneqq \min\{\epsilon_1, \epsilon_2\}$.
We thus complete the proof of Theorem~\ref{thm:multi-limsup} by induction.
\end{proof}

\begin{remark}
If we denote the corresponding constant in Theorem~\ref{thm:uniform-limsup} for each $\Phi_i$ by $\epsilon_i$, then it follows readily from the proof that the $\epsilon$ in Theorem~\ref{thm:multi-limsup} is the minimum of the $\epsilon_i$.
\end{remark}

The example below explains why we have to consider sets $T = \{||\bn|| : \bn \in \bT\}$ of positive density in Theorem~\ref{thm:multi-limsup}, rather than sets $\bT \subseteq \NN^m$ of positive density.\footnote{Generalizing Definition~\ref{def:density}, the {\it upper asymptotic (or natural) density} of $\bT \subseteq \NN^m$ is defined by $\den(\bT) \coloneqq \limsup_{n\to \infty} |\bT \cap [0, n]^m|/(n+1)^m$.
}

\begin{example}
\label{ex:multi-density}
We define two self-maps $\Phi_1$ and $\Phi_2$ of $X=\AA^3$ as follows:
\[
\Phi_1(x,y,z) = (2x,y+1,z) \quad \text{and} \quad \Phi_2(x,y,z) = (xz,y,z+1).
\]
It is easy to verify that $\Phi_1\circ \Phi_2(x,y,z) = \Phi_2\circ \Phi_1(x,y,z) = (2xz, y+1, z+1)$.
Now let us consider the point $(1,0,0)$.
First, for any $n_2\in \ZZ^+$ and $n_1\in \NN$, we have $\Phi_2^{n_2}(1,0,0) = (0,0,n_2)$ and $\Phi_1^{n_1}(0,0,n_2) = (0,n_1,n_2)$.
So,
\[
\Phi^{n_1,n_2}(1,0,0) =
\begin{cases}
  (0,n_1,n_2) & \text{if } n_2>0, \\
  (2^{n_1},n_1,0) & \text{if } n_2=0.
\end{cases}
\]
Let $f(x,y,z)=x$ be the projection to the $x$-coordinate.
Then $f(\Phi^{n_1,n_2}(1,0,0)) = 0$ if $n_2>0$, and $2^{n_1}$ if $n_2=0$.
It follows that if $\bT\subseteq \NN^2$, then
\[
\limsup_{(n_1,n_2)\in \bT} \frac{h(f(\Phi^{n_1,n_2}(1,0,0)))}{\log(n_1+n_2)} > 0
\]
only if $\bT$ contains infinitely many points from the ray $R \coloneqq \{(n_1,n_2) \in \NN^2 : n_2=0\}$ which has density zero. On the other hand, the $\liminf$ of the above quantity is zero as long as $\bT$ has positive density.
Moreover, the only set $\bT$ over which the $\liminf$ is positive is a subset of the above ray $R$ plus finitely many points.
\end{example}

Before we move to a concluding remark on Theorem~\ref{thm:multi-limsup}, let us recall the precise definition of multivariate $D$-finite power series due to Lipshitz.

\begin{definition}[{\cite{Lipshitz89}}]
\label{def:multi-D-finite}
A formal power series $F(\bz) = F(z_1,\ldots,z_m) \in \CC[[\bz]]$ is said to be {\it $D$-finite} over $\CC(\bz)$, if the set of all derivatives $(\partial/\partial z_1)^{i_1} \cdots (\partial/\partial z_m)^{i_m}$ with $i_j\in \NN$ span a finite-dimensional $\CC(\bz)$-vector space.
Equivalently, for each $i\in \{1,\ldots,m\}$, $F(\bz)$ satisfies a nontrivial linear partial differential equation of the form
\[
\left\{ p_{i,n_i}(\bz)\Big(\frac{\partial}{\partial z_i}\Big)^{n_i} + p_{i,n_i-1}(\bz)\Big(\frac{\partial}{\partial z_i}\Big)^{n_i-1} + \cdots + p_{i,0}(\bz) \right\}F(\bz) = 0,
\]
where the $p_{i,j}(\bz) \in \CC[\bz]$.
\end{definition}

At the end of the introduction, we mentioned that it appears to be a subtle issue to deduce a multivariate $D$-finiteness result (e.g., \cite[Theorem~1.3(c)]{BNZ}) from our Theorem \ref{thm:multi-limsup}. In the univariate case the coefficients of a $D$-finite power series arise as $f(\Phi^n(c))$ for certain choices of $X$, $\Phi$, $f$, and $c$; see \cite[Section~3.2.1]{DML-book}.
However, in Example~\ref{ex:multi-coeff} below, we construct a rational function in two variables whose coefficients never arise as $f(\Phi_1^{n_1}\circ\Phi_2^{n_2}(c))$ for any choices of $X$, $\Phi_1$, $\Phi_2$, $f$, and $c$.
It is well known that all algebraic functions are $D$-finite (see \cite[Proposition~2.3]{Lipshitz89}).

\begin{example}
\label{ex:multi-coeff}
Let us consider the following rational function
\[
F(z_1,z_2) \coloneqq \frac{1}{(1-z_1z_2)(1-z_1)} = \sum_{n_2\ge 0}\sum_{n_1\ge n_2} z_1^{n_1} z_2^{n_2}.
\]
We let $a_{n_1,n_2} = 1$ if $n_1\ge n_2\ge 0$ and $a_{n_1,n_2} = 0$ if $n_2>n_1\ge 0$, so that
\[
F(z_1,z_2) = \sum_{n_1,n_2\in \NN} a_{n_1,n_2}z_1^{n_1} z_2^{n_2}.
\]
Our goal is to show that there is no choice of algebraic variety $X$, commuting rational self-maps $\Phi_1$, $\Phi_2 \colon X\dra X$, rational function $f\colon X \dra \bP^1$ all defined over $\Qbar$, and a point $c \in X_{\Phi_1,\Phi_2,f}(\Qbar)$ such that
\[
f(\Phi_1^{n_1} \circ \Phi_2^{n_2} (c)) = a_{n_1,n_2},
\]
for sufficiently large $n_1$ and $n_2$.

Suppose to the contrary that such choices exist.
First, as in the proof of Theorem~\ref{thm:multi-limsup} or in Step~\ref{step:set-up} of the proof of Theorem~\ref{thm:uniform-limsup}, we can find a Noetherian topological space $U$ such that $\Phi_1$ and $\Phi_2$ are continuous regular self-maps of $U$, $f \colon U \to \bP^1$ is regular and continuous, and $c\in U(\Qbar)$.
Let $Z_{i,j}$ denote the closure of the set of $\{ \Phi_1^{n_1}\circ \Phi_2^{n_2}(c) : n_1\ge i, n_2\ge j\}$ in $U$.
Then, as in the proof of Lemma~\ref{l:truncated-orbit}, there is some $N$ such that $Z_{N,N} = Z_{i,j}$ for all $i>N$ and $j>N$.
Let $c' \coloneqq \Phi_1^N \circ \Phi_2^N(c)$ and note that the orbit $\cO_{\Phi_1,\Phi_2}(c')$ is dense in $Z \coloneqq Z_{N,N}$.

We now choose a sufficiently large $N$ such that
\begin{equation}
\label{eq:f-values}
f(\Phi_1^{n_1} \circ \Phi_2^{n_2}(c')) = 1 \text{ if } n_1\ge n_2\ge 0 \text{ and } 0 \text{ if } n_2>n_1\ge 0.
\end{equation}
Let $T_1$ (resp. $T_0$) denote $f^{-1}(1)\cap Z$ (resp. $f^{-1}(0)\cap Z$).
Then $c' \in T_1$.
Since the orbit of $c'$ is dense in $Z$ and also contained in the closed set $T_0\cup T_1$, we have $Z=T_0\cup T_1$.
Let $T_{0,1},\ldots,T_{0,r}$ denote the irreducible components of $T_0$ and let $T_{1,1},\ldots,T_{1,s}$ denote the irreducible components of $T_1$.
Clearly, for each $T_{0,i}$, the intersection $\cO_{\Phi_1,\Phi_2}(c')\cap T_{0,i} \neq \emptyset$; pick up an arbitrary $c'_i$ in it.
Then by \eqref{eq:f-values} there exists some fixed $L_i\in \ZZ^+$ large enough such that $f(\Phi_1^{L_i}(c'_i)) = 1$, or $\Phi_1^{L_i}(c'_i) \in T_1$.
It thus follows from the irreducibility of $T_{0,i}$ that $\Phi_1^{L_i}(T_{0,i}) \subseteq T_1$.
Let $L$ be the maximum of the finite $L_i$.
By \eqref{eq:f-values} again, we have $f(\Phi_2^{L+1}(c')) = 0$ which yields that $\Phi_2^{L+1}(c') \in T_{0,i}$ for some $i$.
Hence, $\Phi_1^{L_i} \circ \Phi_2^{L+1}(c') \in \Phi_1^{L_i}(T_{0,i}) \subseteq T_1$ by the choice of $L_i$.
On the other hand, as $L_i < L+1$, we have $\Phi_1^{L_i} \circ \Phi_2^{L+1}(c') \in T_0$, which is absurd.
\end{example}


\end{document}